\documentclass[12pt]{article}
\usepackage[margin=2cm]{geometry}

\usepackage{amsmath,amsfonts,amssymb,amsthm}
\usepackage{txfonts}%
\usepackage{bm}
\usepackage{comment}
\usepackage{cases}
\usepackage{comment}

\theoremstyle{plain}
\newtheorem{theorem}{Theorem}[section]
\newtheorem{lemma}[theorem]{Lemma}

\newtheorem{proposition}[theorem]{Proposition}
\theoremstyle{definition}

\theoremstyle{remark}
\newtheorem{remark}[theorem]{Remark}

\renewcommand{\indent}{\hspace*{5mm}}
\newcommand{\cK}{{\cal K}}

\newcommand{\br}{\gamma}  
\newcommand{\uc}{\psi^{U}}  
\newcommand{\lc}{\psi^{L}}  
\newcommand{\cps}{{\cal K}} 
\newcommand{\mcps}{{\cal L}} 
\newcommand{\bss}{{\cal Q}} 
\newcommand{\bssd}{{\cal M}} 
\newcommand{\dynp}{{\cal N}} 

\begin{document}
	
\title{Erd\H{o}s-Feller-Kolmogorov-Petrowsky law of the iterated logarithm for
self-normalized martingales: \\ a game-theoretic approach}

\author{
	Takeyuki Sasai\thanks{Graduate School of Information Science and Technology, University of Tokyo}, \ 
	Kenshi Miyabe\thanks{School of Science and Technology, Meiji University} \ 
	and Akimichi Takemura\footnotemark[1]\ %
}
\date{April, 2015}
\maketitle

\begin{abstract}
We prove an Erd\H{o}s--Feller--Kolmogorov--Petrowsky law of the iterated logarithm for self-normalized
martingales. Our proof is given in the framework of the
game-theoretic probability of Shafer and Vovk. As many other game-theoretic proofs, our
proof is self-contained and explicit.
\end{abstract}


\section{Main Result}
Let $S_n$ be a martingale with respect to a filtration $\{{\cal F}_n\}_{n=0}^\infty$ and
let $x_n=S_n - S_{n-1}$ be the martingale difference. On some regularity conditions
on the growth of $|x_n|$, various versions of the law of the iterated logarithm (LIL) have been given in literature. In particular the 
Erd\H{o}s--Feller--Kolmogorov--Petrowsky law of the iterated logarithm (EFKP-LIL \cite[Chapter 5.2]{Revesz2013Random}) is an important extension of LIL.
Erd\H{o}s \cite{Erdos1942Law} proved EFKP-LIL for symmetric Bernoulli random variables.
EFKP-LIL has been generalized by Feller \cite{Feller1943General} for bounded and independent random variables and \cite{Feller1946Law} (see also Bai \cite{Bai1989Theorem}) for the i.i.d.\ case.
Further, EFKP-LIL has been generalized for martingales by Strassen \cite{Strassen1967Almost}, Jain, Jogdeo and Stout \cite{JainJogdeoStout1975Upper}, Philipp and Stout \cite{PhilippStout1986Invariance}, Einmahl and Mason \cite{EinmahlMason1990Some} and Berkes, H{\"o}rmann and Weber \cite{BerkesHormannWeber2010Upper}. 
In particular, Einmahl and Mason \cite{EinmahlMason1990Some} proved a  martingale analogue 
of Feller's result in \cite{Feller1943General}, 
just as Stout \cite{Stout1970Martingale} obtained a martingale analogue of Kolmogorov's result in \cite{Kolmogoroff1929Uber}.

For self-normalized processes, EFKP-LIL was derived by \cite{GriffinKuelbs1991Some,CsorgoSzyszkowiczWang2003Darling} in the i.i.d.\ case.
However EFKP-LIL has not been derived in the martingale case, even though de la Pe{\~n}a, Klass and Lai \cite{PenaKlassLai2004Self} obtained the usual LIL. The purpose of this paper is to prove 
EFKP-LIL for self-normalized martingales.
For a positive non-decreasing continuous function $\psi(\lambda)$ let
\begin{align}
\label{eq:Ipsi}
I(\psi):=\int_1^\infty \psi(\lambda) e^{-\psi(\lambda)^2/2} \frac{d\lambda}{\lambda}.
\end{align}
We state our main theorem.

\begin{theorem}
	\label{th:m-self-normalized-efkp-lil}
	Let $S_n,\,n=1,2,\ldots,$ be a martingale with $S_0 = 0$ and $x_n = S_n -S_{n-1}$ be a martingale difference with respect to a filtration $\{\mathcal{F}_{n}\}_{n=0}^{\infty}$ such that 
	\begin{align*}
	|x_n| \le c_n\,\,a.s.
	\end{align*}
	for some $\mathcal{F}_{n-1}$-measurable random variable $c_n$. Let 
\[
A_n^2 := \sum_{i=1}^n x_i^2 \ \ge 0
\]
and let $\psi$ be a positive non-decreasing continuous function. 

If $I(\psi)< \infty$, then
\begin{equation}
{\mathrm{P} }\left( S_n < A_n \psi(A_n^2) \ a.a. \mid \lim A_n = \infty, \limsup c_n \frac{\psi(A_n^2)^3}{A_n} < \infty \right) = 1.
\label{eq:measure-validity}
\end{equation}
If $I(\psi)= \infty$, then
\begin{equation}
{\mathrm{P} }\left( S_n \ge A_n \psi(A_n^2) \ i.o. \mid \lim A_n = \infty, \limsup c_n \frac{\psi(A_n^2)^3}{A_n} < \infty \right) = 1.
\label{eq:measure-sharpness}
\end{equation}
\end{theorem}

This theorem is a self-normalization of the result in Einmahl and Mason \cite{EinmahlMason1990Some}
and  a generalization of the result in de la Pe{\~n}a, Klass and Lai \cite{PenaKlassLai2004Self}.
The order of growth $A_n/(\psi(A_n^2))^3$ for $c_n$ is currently the best known order for EFKP-LIL
even in the independent case (\cite{BerkesHormannWeber2010Upper}).
We call \eqref{eq:measure-validity} the {\em validity}  and
\eqref{eq:measure-sharpness} the {\em sharpness} of EFKP-LIL.

In \eqref{eq:measure-validity} and \eqref{eq:measure-sharpness}, we 
are not assuming that the conditioning events happen with probability one.  We can state
\eqref{eq:measure-validity} equivalently as
\begin{equation}
{\mathrm{P} }\left( \lim A_n = \infty, \limsup c_n \frac{\psi(A_n^2)^3}{A_n} < \infty,
 S_n \ge  A_n \psi(A_n^2) \ i.o. \right) = 0.
\label{eq:measure-validity1}
\end{equation}
For our proof we adopt the framework of game-theoretic probability by Shafer and Vovk 
\cite{ShaferVovk2001Probability}.  In a game-theoretic approach, for proving  
\eqref{eq:measure-validity}, we 
explicitly construct a non-negative martingale diverging to infinity on the event of 
\eqref{eq:measure-validity1}.

We use the following notation throughout the paper
\begin{align*}
\ln_k n := \underbrace{\ln \ln \dots \ln}_{k \text{times}} n.
\end{align*}
We also fix a small positive $\delta$ for the rest of this paper, e.g., $\delta=0.01$.
For our proof, as is often seen in the upper-lower class theory (cf.\ Feller \cite[Lemma 1]{Feller1946Law}),
we can restrict our attention to $\psi$ such that
\begin{align}
	\label{eq:uc0}
	\lc(n)\le\psi(n)\le\uc(n)\mbox{ for all sufficiently large }n,
\end{align}
where
\[
\lc(n):=\sqrt{2 \ln_2 n + 3 \ln_3 n}, \quad
\uc(n):=\sqrt{2 \ln_2 n + 4 \ln_3 n}.
\]
Here $L$ means the lower class and $U$ means the upper class.
It can be verified that $I(\uc)<\infty$ and $I(\lc)=\infty$.

The rest of this paper is organized as follows.
In Section \ref{sec:gtp} we give a game-theoretic statement  corresponding to our main theorem.
In Section \ref{sec:validity} we give a proof of the validity and
in Section \ref{sec:sharpness} we give a proof of the sharpness.

\section{Preliminaries on Game-Theoretic Probability}
\label{sec:gtp}
In order to state a game-theoretic version of Theorem \ref{th:m-self-normalized-efkp-lil}, 
consider the following simplified predictably unbounded forecasting game (SPUFG, Section 5.1 of \cite{ShaferVovk2001Probability}) with the initial capital $\alpha>0$.
\begin{quote}
	{\sc Simplified Predictably Unbounded Forecasting Game}\\
	\textbf{Players}: Forecaster, Skeptic, Reality\\
	\textbf{Protocol}:\\
	\indent $\cps_0:=\alpha$.\\
	\indent FOR $n=1,2,\ldots$:\\
	\indent\indent Forecaster announces $c_n \ge 0$.\\
	\indent\indent Skeptic announces $M_n\in\mathbb{R}$.\\
	\indent\indent Reality announces $x_n\in[-c_n,c_n]$.\\
	\indent\indent $\cps_n:=\cps_{n-1}+M_n x_n$.\\
	\textbf{Collateral Duties}:
	Skeptic must keep $\cps_n$ non-negative.
	Reality must keep $\cps_n$ from tending to infinity.
\end{quote}
Usually $\alpha$ is taken to be 1, but in Section \ref{sec:sharpness}  we use $\alpha\neq 1$ for notational simplicity.

We prove the following theorem, which implies Theorem \ref{th:m-self-normalized-efkp-lil}
by Chapter 8 of \cite{ShaferVovk2001Probability}.
\begin{theorem}
	\label{th:self-normalized-efkp-lil} Consider SPUFG.
	Let $\psi$ be a positive non-decreasing continuous function. 
	If $I(\psi)<\infty$, Skeptic can force
	\begin{align}
	A_n^2 \rightarrow \infty \,\, \text{and}\,\, \limsup c_n \frac{\psi(A_n^2)^3}{A_n} <\infty
	\ \Rightarrow \ S_n < A_n\psi(A_n^2) \ \  a.a.  
        \label{eq:validity1}
	\end{align}
	and if $I(\psi)=\infty$, Skeptic can force
	\begin{align}
	A_n^2 \rightarrow \infty \,\, \text{and} \,\, &\limsup c_n \frac{\psi(A_n^2)^3}{A_n} <\infty
	\ \Rightarrow \  S_n \ge A_n\psi(A_n^2) \ \  i.o.		
	\label{eq:sharpness1}
	\end{align}
\end{theorem}
We use the same line of arguments as in  \cite{MiyabeTakemura2013Law} and Chapter
5 of Shafer and Vovk \cite{ShaferVovk2001Probability}.
We employ a Bayesian mixture of constant-proportion betting strategies.
Here we give basic properties of constant-proportion betting strategies.

A constant-proportion betting strategy with betting proportion $\br>0$ sets
\begin{align*}
M_n = \br \cps_{n-1}.
\end{align*}
However, $\cps_n$ becomes negative if $\br x_n< -1$.
For simplicity we consider applying the strategy (``keep the account open'')
as long as $\br c_n \le \delta$ and sets $M_n=0$ once $\br c_n > \delta$ happens (``freeze the account'').
Define a stopping time
\begin{align}
\label{eq:stop}
\sigma_{\br} := \min\{n \mid \br c_n  >\delta\}.
\end{align}
Note the monotonicity of $\sigma_{\br}$, i.e., $\sigma_{\br'}\ge \sigma_{\br}$ if $\br' \le \br$.
We denote the capital process of the constant-proportion betting strategy with this stopping time by $\cps^\br_n$. 
With the initial capital of $\cps^\br_0 = \alpha$, the value of $\cps^\br_n$ is  written as
\begin{align*}
\label{eq:sncps}
\cps^\br_n = \alpha \prod_{i=1}^{\min(n,\sigma_{\br}-1)} (1+\br x_i).
\end{align*}

By 
\[
t - \frac{t^2}{2} - t^2 \times |t|
\le 
\ln(1+t)
\le
t - \frac{t^2}{2} + t^2 \times |t|
\]
for $|t| \le \delta$, taking the logarithm of $\prod_{i=1}^n (1+\br x_i)$, for $n< \sigma_\br$, we have 
\begin{equation*}
\br S_n - \frac{\br^2 A_n^2}{2} - \br^3 A_n^2  \bar c_n
\le
\ln \left(\cps_n^\br /\alpha\right)
\le 
\br S_n - \frac{\br^2 A_n^2}{2} + \br^3 A_n^2 \bar c_n
\end{equation*}
and
\begin{equation}
\label{eq:cp-bound}
e^{- \br^3 A_n^2  \bar c_n}  e^{\br S_n - \br^2 A_n^2/2} 
\le 
\cps^\br_n / \alpha 
\le
e^{\br^3 A_n^2  \bar c_n }  e^{\br S_n - \br^2 A_n^2/2},
\end{equation}
where 
\[
\bar c_n := \max_{1\le i \le n} c_i.
\]

We also set up some notation for expressing the condition in \eqref{eq:validity1} and \eqref{eq:sharpness1}.
An infinite sequence of  Forecaster's and Reality's announces $\omega = (c_1,x_1,c_2,x_2,\ldots)$ is called a \textit{path} and the set of paths $\Omega=\{\omega\}$ is called the sample space.
Define a subset $\Omega_{<\infty}$ of $\Omega$ as
\begin{equation*}
\label{eq:sample-space-0}
\Omega_{<\infty} := \left\{ \omega \mid A_n^2 \rightarrow \infty, \limsup_n c_n \frac{\psi(A_n^2)^3}{A_n} < \infty \right\}.
\end{equation*}
For an arbitrary path $\omega \in \Omega_{<\infty}$ we have
\begin{align}
	\label{eq:sn-forecaster-v}
	\exists C(\omega) < \infty,\exists n_1(\omega),\forall n>n_1(\omega),\, c_n<C(\omega)\frac{A_n}{\psi(A_n^2)^3}, \ \psi(A_n^2)\ge 1.
\end{align}
The last inequality holds by the lower bound in \eqref{eq:uc0}.

\section{Validity}
\label{sec:validity}
We prove the validity in  \eqref{eq:validity1} of 
Theorem \ref{th:self-normalized-efkp-lil}.
In this section we let $\alpha=1$.
We discretize the integral in \eqref{eq:Ipsi}
as
\begin{align}
	\label{eq:val-suml-cond}
	\sum_{k=1}^\infty \frac{ \psi(k)}{k} e^{-\psi(k)^2/2} < \infty.
\end{align}
Since $xe^{-x^2/2}$ is decreasing for $x\ge1$,
the function $\lambda\mapsto\frac{\psi(\lambda)}{\lambda}e^{-\psi(\lambda)^2/2}$ is decreasing for $\lambda$ such that $\psi(\lambda)\ge 1$ 
and 
convergences of the integral in \eqref{eq:Ipsi}
and the sum in \eqref{eq:val-suml-cond} are equivalent.


The convergence of the infinite series in \eqref{eq:val-suml-cond} implies the existence of a non-decreasing sequence of positive reals $a_k$ diverging to infinity ($a_k\uparrow \infty$), such that
the series multiplied term by term by $a_k$ is still convergent:
\begin{align*}
	Z:=\sum_{k=1}^\infty a_k\frac{ \psi(k)}{k} e^{-\psi(k)^2/2}  < \infty.
\end{align*}
This is easily seen by dividing the infinite series into blocks of sums less than
or equal to $1/2^k$ and multiplying the $k$-th block by $k$
(see also \cite[Lemma 4.15]{MiyabeTakemura2012Convergence}).


For $k\ge 1$ let 
\[
p_k := \frac{1}{Z}a_k \frac{\psi(k)}{k} e^{-\psi(k)^2/2} 
\]
and consider the capital process of a countable mixture of constant-proportion strategies 
\begin{align}
	\label{eq:validity-br}
	\cK_n := \sum_{k=1}^\infty p_k \cps_n^{\br_k},
	\quad\mbox{ where }\quad
	\br_k := \frac{\psi(k)}{\sqrt{k}}. 
\end{align}
Note that $\cK_n$ is never negative.
By the upper bound in \eqref{eq:uc0}, as $k\rightarrow\infty$ we have
\begin{equation}
\label{eq:gamma-k-zero}
\br_k \le \frac{\uc(k)}{\sqrt{k}} = \sqrt{\frac{2 \ln_2 k + 4 \ln_3 k}{k}} \rightarrow 0. 
\end{equation}

We show that $\limsup_n \cK_n=\infty$  if a path $\omega\in \Omega_{<\infty}$ satisfies
$S_n \ge A_n\psi(A_n^2)\,\,i.o.$ \ 
We bound $Z \cK_n$ as
\begin{equation}
\label{eq:zk-1}
Z\cK_n \ge \sum_{k=\lfloor A^2_n-A^2_n/\psi(A^2_n)\rfloor}^{\lfloor A^2_n \rfloor} p_k \cK_n^{\br_k}
.
\end{equation}
At this point we check that  all accounts on  
the right-hand side of \eqref{eq:zk-1} are open for sufficiently large $n$  and 
the lower bound in 
\eqref{eq:cp-bound} can be applied to each term of \eqref{eq:zk-1} for $\omega\in \Omega_{<\infty}$.
We have the following two lemmas.

\begin{lemma}
\label{lem:bar_c_n_bound}
Let $\omega\in \Omega_{<\infty}$. Let $C=C(\omega)$ in \eqref{eq:sn-forecaster-v}.
For sufficiently large $n$ 
\begin{equation}
\label{eq:bar_c_n_bound}
\bar c_n = \max_{1\le i \le n} c_i < (1+\delta) C  \frac{A_n}{\psi(A_n^2)^3}.
\end{equation} 
\end{lemma}

\begin{proof}
Note that the first $n_1(\omega)$ $c$'s i.e., $c_1, \dots, c_{n_1(\omega)}$,
do not matter since  $\lim_{n\rightarrow\infty} A_n/\psi(A_n^2)^3=\infty$.
For $l >  n_1(\omega)$, by \eqref{eq:sn-forecaster-v} we have
\[
c_l \le C \frac{A_l}{\psi(A_l^2)^3} \le C A_l.
\]
Hence $c_l$ such that 
$A_l \le  A_n/{\psi(A_n^2)^3}$ do not matter in $\bar c_n$.

For $c_l$ such that $A_l >  A_n/{\psi(A_n^2)^3}$ we have
\[
c_l \le C \frac{A_l}{\psi\big(A_n^2/\psi(A_n^2)^6\big)^3} \le
C \frac{A_n}{\psi\big(A_n^2/\psi(A_n^2)^6\big)^3} 
= C \frac{A_n}{\psi(A_n^2)^3} \frac{\psi(A_n^2)^3}{\psi\big(A_n^2/\psi(A_n^2)^6\big)^3}.
\]
But by \eqref{eq:uc0}, 
both $\psi(A_n^2)$ and 
$\psi\big(A_n^2/\psi(A_n^2)^6\big)$ are of the order 
$\sqrt{2 \ln_2 A_n^2}(1+o(1))$ and 
$\psi(A_n^2)/\psi\big(A_n^2/\psi(A_n^2)^6\big) \rightarrow 1$ as $n\rightarrow\infty$.
Hence \eqref{eq:bar_c_n_bound} holds.
\end{proof}

\begin{lemma}
\label{lem:keep-open}
Let $\omega\in \Omega_{<\infty}$. 
For sufficiently large $n$, $\sigma_{\br_k} > n$ for all 
$k=\lfloor A^2_n-A^2_n/\psi(A^2_n)\rfloor, \dots, \lfloor A^2_n \rfloor$.
\end{lemma}
\begin{proof}
By the monotonicity of $\psi$,  we have $\br_k \le  \psi(A_n^2)/\sqrt{\lfloor A^2_n-A^2_n/\psi(A^2_n)\rfloor}$ 
for $k=\lfloor A^2_n-A^2_n/\psi(A^2_n)\rfloor, \dots, \lfloor A^2_n \rfloor$.
Then by the monotonicity of $\sigma_\br$,
it suffices to show 
\[
\frac{\psi(A_n^2)}{\sqrt{\lfloor A^2_n-A^2_n/\psi(A^2_n)\rfloor}}
\bar c_n \le \delta
\]
for sufficiently large $n$. By \eqref{eq:bar_c_n_bound}, the left-hand side is bounded from above by
\[
\frac{\psi(A_n^2)}{\sqrt{\lfloor A^2_n-A^2_n/\psi(A^2_n)\rfloor}}
\times (1+\delta) C  \frac{A_n}{\psi(A_n^2)^3} = 
(1+\delta) C \frac{A_n}{\sqrt{\lfloor A^2_n-A^2_n/\psi(A^2_n)\rfloor}}
\frac{1}{\psi(A_n^2)^2}.
\]
But this converges to 0 as $n\rightarrow\infty$.
\end{proof}


By Lemma \ref{lem:keep-open} and the lower bound in 
\eqref{eq:cp-bound},  for sufficiently large $n$, we have
\begin{align*}
	\cps^{\br_{k}}_n  
	\ge
	e^{-\br_{k}^3 A_n^2 \bar c_n}  e^{\br_k S_n - \br_{k}^2 A_n^2/2},\quad k=\lfloor A_n^2-A_n^2/\psi(A_n^2)\rfloor,\dots, \lfloor A_n^2\rfloor
\end{align*}
and $Z \cK_n$ can be evaluated from below as
\allowdisplaybreaks
\begin{align*}
	Z \cK_n 
	&\ge
	Z\sum_{k=\lfloor A^2_n-A^2_n/\psi(A^2_n)\rfloor}^{\lfloor A^2_n \rfloor} p_k \exp( \br_k S_n  - \frac{\br_k^2A_n^2}{2} - \br_k^3 A_n^2 \bar c_n )\\
	&= 
	\sum_{k=\lfloor A^2_n-A^2_n/\psi(A^2_n)\rfloor}^{\lfloor A^2_n \rfloor}  a_k \frac{\psi(k)}{k} \exp(-\frac{\psi(k)^2}{2} +  \br_k S_n - \frac{\br_k^2A_n^2}{2} - \br_k^3 A_n^2 \bar c_n )
\end{align*}

Now we assume that $S_n \ge A_n \psi(A_n^2)\ i.o.$ for the path $\omega\in \Omega_{<\infty}$. 
Then for sufficiently large $n$ such that $S_n \ge A_n \psi(A_n^2)$, 
 $\psi(A_n^2)/(\psi(A_n^2) -1) \le 1+ \delta$ and $A_n/\left(\lfloor A_n^2-A_n^2/\psi(A_n^2)\rfloor \right)^{1/2} \le 1+ \delta$, 
we evaluate the exponent part by \eqref{eq:cp-bound} as
\begin{align*}
	-\frac{\psi(k)^2}{2}+ \br_k S_n - \frac{\br_k^2 A_n^2}{2}
	&\ge
	-\frac{\psi(k)^2}{2}+A_n\psi(A_n^2)\frac{\psi(k)}{\sqrt{k}}-\frac{\psi(k)^2}{k}\frac{A_n^2}{2}\\
	&=
	\psi(k)\left(-\frac{1}{2}\left(1+\frac{A_n^2}{k}\right)\psi(k)+\sqrt{\frac{A_n^2}{k}}\psi(A_n^2)\right)\\
	&\ge
	-\frac{\psi(A_n^2)^2}{2}\left(\sqrt{\frac{A_n^2}{k}}-1\right)^2\ge-\frac{\psi(A_n^2)^2}{2}\left(\frac{A_n^2}{k}-1\right)^2 \\
	&\ge
	-\frac{1}{2}\left(\frac{\psi(A_n^2)}{\psi(A_n^2)-1}\right)^2\ge-\frac{1}{2}-2\delta
\end{align*}
and by Lemma \ref{lem:bar_c_n_bound}
\begin{align}
	\br_k^3 A_n^2 \bar c_n 	&\le
	\frac{\psi(A_n^2)^3}{\left(\lfloor A_n^2-A_n^2/\psi(A_n^2)\rfloor \right)^{3/2}}  A_n^2 (1+\delta) C\frac{A_n }{\psi(A_n^2)^3} \nonumber  \\
	&\le
	(1+\delta)C 
        \left(\frac{A_n}{\left(\lfloor A_n^2-A_n^2/\psi(A_n^2)\rfloor \right)^{1/2}} \right) ^3
        \nonumber \\
	&\le
	C(1+ \delta)^4.
	\label{eq:validity-deltaconst}
\end{align}

For sufficiently large $n$, we have
\begin{align*}
	\psi(A_n^2)\le\uc(A_n^2)<\uc(2k)=\sqrt{2\ln_2 2k+4\ln_3 2k}
	<2\sqrt{2\ln_2 k+3\ln_2 k}=2\lc(k)\le2\psi(k).
\end{align*}
Thus by \eqref{eq:validity-deltaconst},
\begin{align*}
	Z \cK_n 
	&\ge 
	\sum_{k=\left\lfloor A_n^2-A_n^2/\psi(A_n^2) \right\rfloor}^{\lfloor A^2_n \rfloor } a_k \frac{\psi(k)}{k}\exp\left(-\frac{1}{2}-2\delta - C(1+\delta)^4 \right)\\
	&\ge
	a_{\left\lfloor A_n^2-A_n^2/\psi(A_n^2) \right\rfloor} \frac{\psi(A_n^2)}{2A^2_n} \sum_{k=\lfloor A_n^2-A_n^2/\psi(A_n^2) \rfloor}^{\lfloor A^2_n \rfloor} \exp \left(-\frac{1}{2}-2\delta - C(1+\delta)^4 \right)\\
	&\ge
	a_{\left\lfloor A_n^2-A_n^2/\psi(A_n^2) \right\rfloor} \frac{\psi(A_n^2)}{2A^2_n} \left(\frac{A^2_n}{\psi(A_n^2)}-1\right)\exp\left(-\frac{1}{2}-2\delta - C(1+\delta)^4 \right)\\
	&= 
	a_{\left\lfloor A_n^2-A_n^2/\psi(A_n^2) \right\rfloor} \left(\frac{1}{2}-\frac{\psi(A_n^2)}{2A^2_n}\right)\exp \left(-\frac{1}{2}-2\delta - C(1+\delta)^4 \right).
\end{align*}
Since $a_{\left\lfloor A_n^2-A_n^2/\psi(A_n^2) \right\rfloor}\rightarrow\infty$ as $n \rightarrow \infty$, we have shown
\[
\omega\in \Omega_{<\infty}, \ S_n \ge A_n\psi(A_n^2)  \ i.o.\ \Rightarrow \ \limsup_{n\rightarrow\infty} \cK_n = \infty.
\]

\section{Sharpness}
\label{sec:sharpness}
We prove the sharpness in \eqref{eq:sharpness1} of Theorem \ref{th:self-normalized-efkp-lil}.
As in Section 4.2 of \cite{ShaferVovk2001Probability}
and in \cite{MiyabeTakemura2012Convergence}, in order to prove the sharpness, 
it suffices to show the following proposition.
\begin{proposition}
	\label{th:self-normalized-efkp-lil-dash}  Consider SPUFG.
	Let $\psi$ be a positive non-decreasing continuous function.  If $I(\psi)=\infty$, then for each $C>0$,
Skeptic can force
	\begin{align}
	\label{eq:sn-forecaster-s}
A_n^2 \rightarrow \infty, \limsup_n c_n \frac{\psi(A_n^2)^3}{A_n} \le C  \ \Rightarrow \ S_n \ge A_n\psi(A_n^2) \ \  i.o.	\end{align}
      \end{proposition}
Once we prove this proposition, 
we can take the mixture over $C=1,2,\dots$. Then the sharpness follows,
because for each $\omega\in \Omega_{<\infty}$, there exists $C(\omega)$ satisfying
\eqref{eq:sn-forecaster-v}.  We denote
\begin{align*}
\Omega_C 
&:=  \left\{ \omega\in \Omega \mid A_n^2 \rightarrow \infty, \limsup_n c_n \frac{\psi(A_n^2)^3}{A_n} < (1-\delta)C\right\},\\
\Omega_0 
&:=  \left\{ \omega\in \Omega \mid \lim_{n \rightarrow \infty}A_n^2  < \infty\right\},\\
\Omega_{=\infty}
&:= \left\{ \omega\in \Omega \mid A_n^2 \rightarrow \infty, \limsup_n c_n \frac{\psi(A_n^2)^3}{A_n} = \infty \right\}.
\end{align*}
We divide our proof of Proposition \ref{th:self-normalized-efkp-lil-dash}
into several subsections.
For notational simplicity we use the initial capital of $\alpha=1-2/e=(e-2)/e$ in this section.
In Sections \ref{subsec:uniform-mixture} and \ref{subsec:bss} we only consider $\br$ and $n$ with $n <  \sigma_\br$.  As in Lemma \ref{lem:keep-open} for the validity, 
this condition will be satisfied for sufficiently small $\br$ and relevant $n$.

\subsection{Uniform mixture of constant-proportion betting strategies}
\label{subsec:uniform-mixture}
We consider a continuous uniform mixture of constant-proportion strategies with 
the betting proportion $u\br$, $2/e\le u\le 1$. This is a Bayesian strategy, a similar one to which has been considered in \cite{KumonTakemuraTakeuchi2008Capital}. 

Define
\begin{align*}
\mcps^\br_n 
:=\int_{2/e}^1 \prod_{i=1}^{\min(n,\sigma_{\br}-1)} (1+u\br x_i)du, \qquad \mcps^\br_0=\alpha=1-e/2.
\end{align*}
At round $n < \sigma_\br$ this strategy bets
$
M_n = \int_{2/e}^1 u\br   \prod_{i=1}^{n-1} (1+u\br x_i) du .
$
Then by induction on $n<\sigma_\br$ the capital process is indeed written as
\begin{align*}
	\mcps^\br_n 
	&= 
	\mcps^\br_{n-1} + M_n x_n =\int_{2/e}^1 \prod_{i=1}^{n-1} (1+u\br x_i) du
	+ x_n\int_{2/e}^1 u\br\prod_{i=1}^{n-1} (1+u\br x_i) du \\ 
	&=   
	\int_{2/e}^1 \prod_{i=1}^n (1+u\br x_i)du.  %
\end{align*}
Applying \eqref{eq:cp-bound},  we have
\begin{align*}
	e^{ -\br^3 A_n^2 \bar c_n } \int_{2/e}^1 e^{u\br S_n - u^2\br^2 A_n^2/2}du 
	\le
	\mcps^\br_n 
	\le  
	e^{ \br^3 A_n^2 \bar c_n }\int_{2/e}^1 e^{u\br S_n - u^2\br^2 A_n^2/2}  du,
\end{align*}
for $n <  \sigma_{\br}$.
We further bound the integral in the following lemma.
\begin{lemma}
	\label{lem:mcps-upper} %
For $n < \sigma_{\br}$,
	\begin{numcases}
		{\mcps^\br_n\le }
		e^{\br^3 A_n^2 \bar c_n} e^{2\br(S_n/e - \br A_n^2/e^2)}
		& \text{if}\quad $S_n\le 2\br A_n^2/e$,
		\label{eq:mcps-positive1-sn}\\
		e^{\br^3 A_n^2 \bar c_n}\min \left\{ e^{S_n^2/(2A_n^2)} \frac{\sqrt{2\pi}}{\br A_n}, e^{\br S_n/2} \right\}
		&  \text{if}\quad $2\br A_n^2/e<S_n< \br A_n^2$,
		\label{eq:mcps-positive2-sn}
		\\
		\label{eq:mcps-positive3-sn}
		e^{\br^3 A_n^2 \bar c_n}\min\left\{e^{S_n^2/(2A_n^2)}\frac{\sqrt{2\pi}}{\br A_n},e^{\br S_n-\br^2 A_n^2/2}\right\}& \text{if}\quad $S_n \ge  \br A_n^2$.
	\end{numcases}
\end{lemma}

\begin{proof}
	Completing the square we have
	\begin{equation*}
		- \frac{1}{2}u^2\br^2 A_n^2  + u\br S_n = -\frac{\br^2 A_n^2}{2}  \left(u-\frac{S_n}{\br A_n^2}\right)^2 + 
		\frac{S_n^2}{2A_n^2}.
	\end{equation*}
	Hence  by the change of variables
	\[
	v = \br A_n \left( u-\frac{S_n}{\br A_n^2}\right), \qquad  du = \frac{dv}{\br A_n},
	\]
	we obtain
	\begin{align*}
		\int_{2/e}^1 e^{u \br S_n - u^2\br^2 A_n^2/2}  du 
		&=
		e^{S_n^2/(2A_n^2)}
		\int_{2/e}^1 \exp \left(-\frac{\br^2 A_n^2}{2}  \left(u-\frac{S_n}{\br A_n^2}\right)^2\right) du\nonumber \\
		&= 
		e^{S_n^2/(2A_n^2)}\frac{1}{\br A_n} \int_{2 \br A_n/e-S_n/ A_n }^{\br A_n - S_n/A_n}
		e^{-v^2/2} dv .
	\end{align*}
	Then for all cases we can bound $\mcps^\br_n$ from above as
	\begin{equation}
		\label{eq:Q-bound-up-1}
		\mcps^\br_n \le  e^{\br^3 A_n^2 \bar c_n + S_n^2/(2A_n^2)} \frac{\sqrt{2\pi}}{\br A_n}.
	\end{equation}

	Without change of variables, we can also bound 
	the integral $\int_{2/e}^1 g(u)du$, $g(u):=e^{u \br S_n - u^2\br^2 A_n^2/2}$, directly as
	\[
	\int_{2/e}^1 g(u)du \le
	\max_{2/e\le u\le 1} g(u).
	\]
	Note that 
	\begin{equation}
		g(2/e)=e ^{2 \br (S_n/e - \br A_n /e^2)}, \quad g(1)=e^{ \br S_n-\br^2 A_n^2/2}.
		\label{eq:g-lower-upper}
	\end{equation}
	We now consider the following three cases.
	\begin{description}
		\item[Case 1] $ S_n \le 2 \br A_n^2/e$. \ In this case $S_n/(\br A_n^2) \le 2/e$ and
		by the unimodality of $g(u)$ we have $\max_{2/e\le u \le 1}g(u)= g(2/e)$. Hence 
		\eqref{eq:mcps-positive1-sn} follows from \eqref{eq:g-lower-upper}.
		\item[Case 2] $2 \br A_n^2/e < S_n < \br A_n^2$. \  
		In this case  $\max_{2/e\le u \le 1} g(u)=g(S_n/(\br A_n^2))=e^{S_n^2/(2A_n^2)}$ and
		$\mcps^\br_n 
		\le
		e^{\br^3 A_n^2 \bar c_n} e^{S_n^2/(2A_n^2)}$.
		Furthermore in this case $S_n^2 < \br A_n^2 S_n$ implies $S_n^2/(2A_n^2) <  \br S_n/2$
		and we also have
		\begin{equation}
			\label{eq:qn1-case2a}
			\mcps^\br_n  
			\le
			e^{\br^3 A_n^2 \bar c_n} e^{\br S_n/2}.
		\end{equation}
		By \eqref{eq:Q-bound-up-1} %
		and \eqref{eq:qn1-case2a}, 
		we have \eqref{eq:mcps-positive2-sn}.
		\item[Case 3] $S_n \ge \br A_n^2$.  \ Then $S_n/(\br A_n^2)\ge 1$ %
		and $\max_{2/e\le u \le 1} g(u)=g(1)$. Hence
		\begin{equation}
			\label{eq:qn1-case1}
			\mcps^\br_n  
			\le
			e^{\br^3 A_n^2 \bar c_n} e^{ \br S_n-\br^2 A_n^2/2}. %
		\end{equation}
		By \eqref{eq:Q-bound-up-1} and \eqref{eq:qn1-case1},
		we have \eqref{eq:mcps-positive3-sn}.
	\end{description} 
\end{proof}

\subsection{Buying a process and selling a process}
\label{subsec:bss}
Next we consider the following capital process.
\begin{equation}
	\bss_n^{\br} :=  2 \mcps_n^\br  - \cps_n^{\br e}.
\label{eq:bssn-def}
\end{equation}
This capital process consists of buying two units of $\mcps_n^\br$ and
selling one unit of $\cps_n^{\br e}$. 
This combination of selling and buying is essential in the game-theoretic proof of LIL
in Chapter 5 of \cite{ShaferVovk2001Probability} and \cite{MiyabeTakemura2013Law}.
However, unlike Chapter 5 of \cite{ShaferVovk2001Probability} and \cite{MiyabeTakemura2013Law}, where
a combination of {\em three} capital processes is used, we only combine {\em two} capital processes.

We want to bound $\bss_n^{\br}$ from above. 

\begin{lemma}\label{lem:bss-bound}
	Let 
	\begin{align}
		\label{eq:case-i-negative-condition-1}
		C_1 
		&:= 2 e^{\br^3 A_n^2 \bar c_n} 
		\exp \left(\frac{(2e-1)((1+e^3)\br^3 A_n^2 \bar c_n + \ln 2)}{(e-1)^2}\right).
	\end{align}
	Then for  $n < \sigma_{\br e}$,
	\begin{numcases}
		{\bss_n^\br \le }
		C_1 
		& \text{if}\quad $S_n\le \br A_n^2/e$, \label{eq:bss-positive1-sn}\\
		2 e^{\br^3 A_n^2 \bar c_n}\min \left\{ e^{S_n^2/(2A_n^2)} \frac{\sqrt{2\pi}}{\br A_n}, e^{\br S_n} \right\}
		&  \text{if}\quad $\br A_n^2/e<S_n< e\br A_n^2$,
		\label{eq:bss-positive2-sn}
		\\
		\label{eq:bss-positive3-sn}
		C_1 
		& \text{if}\quad $S_n \ge e \br A_n^2$.
	\end{numcases}
\end{lemma}

\begin{remark}
	\label{rem:2}
	In this lemma, $C_1$ depends on $\bar c_n$, $\br$ and $A_n$ through $\br^3 A_n^2 \bar c_n$.
	However from Section \ref{subsec:Futher discrete mixture of processes} on, we evaluate
	$\br^3 A_n^2 \bar c_n$ from above by a constant. Hence, $C_1$ can be also taken to be
	a constant (cf.\ \eqref{eq:C_1-const}) not depending on $\br$ and $A_n$. 
	Also note that the interval for $S_n$ in
	\eqref{eq:bss-positive2-sn} is larger than the interval in \eqref{eq:mcps-positive2-sn}.
\end{remark}

\begin{proof}
	We bound $\bss_n^\br =2\mcps_n^\br - \cps_n^{\br e}$ from above
	in the following three cases:
	\[
	\text{(i)}\  S_n \le \br A_n^2 /e,  \quad 
	\text{(ii)}\ \br A_n^2/e < S_n < e \br A_n^2, \quad 
	\text{(iii)}\ S_n \ge e  \br A_n^2,
	\]
	\begin{description}
		\item[Case (i)]\ In this case $S_n/e - \br A_n^2 /e^2\le 0$. Hence
		\eqref{eq:bss-positive1-sn} follows from \eqref{eq:mcps-positive1-sn} and $\bss_n^\br \le 2 \mcps_n^\br$.
		
		\item[Case (ii)] 
		We again use $\bss^\br_n\le 2\mcps^\br_n$.
		If $\br A_n^2/e < S_n \le 2 \br A_n^2/e$, then 
		\[
		\frac{S_n}{e} - \frac{\br A_n^2 }{e^2}\le \frac{\br A_n^2}{e^2} \le  \frac{S_n}{e}
		\]
		and 
		$\mcps_n^\br\le e^{\br^3 A_n^2 \bar c_n}e^{2\br S_n/e} \le e^{\br^3 A_n^2 \bar c_n}e^{\br S_n}$
		from \eqref{eq:mcps-positive1-sn}.
		Otherwise \eqref{eq:bss-positive2-sn} follows from 
		\eqref{eq:mcps-positive2-sn} and \eqref{eq:mcps-positive3-sn}.
		

		\item[Case (iii)] \ 
		Since $S_n \ge e A_n^2 \br > A_n^2 \br$, by \eqref{eq:qn1-case1}
		we have $\mcps^\br_n  \le e^{\br^3 A_n^2 \bar c_n} e^{ \br S_n-\br^2A_n^2/2}$
		and 
		\begin{align*}
			\bss_n^\br
			&\le 
			2\mcps_n^\br - \cps_n^{\br e}
			\le 2  e^{\br^3 A_n^2 \bar c_n} e^{ \br S_n-\br^2A_n^2/2} - e^{-\br^3 e^3 A_n^2 \bar c_n}e^{\br e S_n - \br^2 e^2 A_n^2/2}\\
			&= 
			2e^{\br^3 A_n^2 \bar c_n} e^{ \br S_n-\br^2A_n^2/2} 
			\left( 1 - \frac{1}{2} e^{-(1+e^3)\br^3A_n^2 \bar c_n}e^{\br (e-1)S_n - (e^2-1)\br^2A_n^2/2}\right). 
		\end{align*}
		Hence if the right-hand side is non-positive we have $\bss_n^\br \le 0$:
		\begin{align}
			&S_n \ge e A_n^2 \br \ \  \text{and}\ -(1+e^3)\br^3 A_n^2 \bar c_n - \ln 2 + \br(e-1)S_n - \frac{1}{2}(e^2-1)\br^2A_n^2 \ge 0 \nonumber\\
			& \qquad \qquad \Rightarrow \ \ \bss_n^\br \le 0.
			\label{eq:case-iii-1}
		\end{align}
		Otherwise, write  $B_n:=(1+e^3) \br^3 A_n^2 \bar c_n + \ln 2$ and 
		consider the case 
		\begin{align*}
		\br(e-1)S_n - \frac{1}{2}(e^2-1) \br^2 A_n^2  \le  B_n.
		\end{align*}
		Dividing this by $e-1$ 
		and also considering $S_n \ge e A_n^2\br$, we have
		\begin{align}
			\label{eq:case-iii-1a}
			\br S_n - \frac{1}{2}(e+1) \br^2A_n^2 &\le  \frac{B_n}{e-1},\\
			-S_n +eA_n^2 \br &\le 0. 
			\label{eq:case-iii-1b}
		\end{align}
		$\br \times \eqref{eq:case-iii-1b} + \eqref{eq:case-iii-1a}$ gives
		\begin{align*}
		\frac{1}{2}(e-1) \br^2 A_n^2 \le \frac{B_n}{e-1} \quad \text{or} \quad \frac{1}{2}  \br^2 A_n^2 \le \frac{B_n}{(e-1)^2}.
		\end{align*}
		Then by \eqref{eq:case-iii-1a}
		\begin{align*}
		\br S_n - \frac{1}{2}\br^2 A_n^2 \le \frac{B_n}{e-1} + \frac{e}{2} \br^2 A_n^2 \le
		\frac{B_n}{e-1} + \frac{eB_n}{(e-1)^2}=\frac{(2e-1)B_n}{(e-1)^2}.
		\end{align*}
		Hence just using $\bss_n^\br\le 2\mcps_n^\br$ and \eqref{eq:qn1-case1} in this case, we obtain
		\begin{align}
			\label{eq:case-iii-conclusion}
			\bss_n^\br  \le 2 e^{\br^3 A_n^2 \bar c_n} \exp\left(\frac{(2e-1)((1+e^3) \br^3A_n^2 \bar c_n + \ln 2 )}{(e-1)^2}\right) = C_1.
		\end{align}
		This also covers \eqref{eq:case-iii-1}
		and we have \eqref{eq:case-iii-conclusion} for the whole case (iii).
	\end{description}
\end{proof}

\subsection{Change of time scale and dividing the rounds into cycles}
\label{subsec:Change of time scale}
For proving the sharpness we consider the change of time scale 
from $\lambda$ to $k$:
\[
\lambda= e^{5k \ln k} = k^{5k}.
\]
By taking the derivative of  $\ln \lambda= 5 k \ln k$, we have
$
d\lambda/\lambda =  5(\ln k+1)dk. 
$
Since $\ln k$ is dominant in $(\ln k+1)$, 
the integrability condition is written as
\begin{equation*}
	\int_1^\infty \psi(\lambda) e^{-\psi(\lambda)^2/2} \frac{d\lambda}{\lambda}  = \infty
	\  \Leftrightarrow \ 
	\int_1^\infty (\ln k) \psi(e^{5k\ln k}) e^{-\psi(e^{5k\ln k})^2/2} dk = \infty.
\end{equation*}
Let $f(x):=\psi(e^{5x\ln x}) e^{-\psi(e^{5x\ln x})^2/2}$.
Since $xe^{-x^2/2}$ is decreasing for $x\ge 1$,
the function $f(x)$ is decreasing for $x$ such that $\psi(e^{5x\ln x})\ge 1$.
Thus, for sufficiently large $k$ and $x$ such that $k\le x\le k+1$, we have
\[
\frac{1}{2}\ln (k+1) f(k+1) \le  \ln k f(x+1)\le  \ln x f(x) \le  \ln (k+1) f(x)\le 2\ln k f(k).
\]
Hence, we have
\begin{equation*}
	\int_1^\infty (\ln k) \psi(e^{5k\ln k}) e^{-\psi(e^{5k\ln k})^2/2} dk = \infty 
	\ \ \Leftrightarrow \ \ 
	\sum_{k=1}^\infty  (\ln k) \psi(e^{5k\ln k}) e^{-\psi(e^{5k\ln k})^2/2} = \infty .
\end{equation*}
Then, it suffices to show \eqref{eq:sn-forecaster-s} if $\sum_{k=1}^\infty  (\ln k) \psi(e^{5k\ln k}) e^{-\psi(e^{5k\ln k})^2/2} = \infty.$

As in Chapter 5 of \cite{ShaferVovk2001Probability} and \cite{MiyabeTakemura2013Law},
we divide the time axis into  ``cycles''.  However, unlike 
in Chapter 5 of \cite{ShaferVovk2001Probability} and \cite{MiyabeTakemura2013Law},
our cycles are based on stopping times.
Let
\begin{equation}
\label{eq:nk-sharpness}
n_k := k^{5k}, \quad k=1,2,\dots,
\end{equation}
and define a family of stopping times
\begin{equation}
\label{eq:stopping-time-cycle}
\tau_k:=
\min \left\{n \mid A^2_n \ge n_k \right\}. 
\end{equation}
We define the $k$-th cycle by $[\tau_k, \tau_{k+1}]$, $k\ge 1$.
Note that $\tau_k$ is finite for all $k$ if and only if $A_n^2 \rightarrow\infty$.
Betting strategy for the $k$-th cycle is based on the following betting proportion:
\begin{equation}
\label{eq:br_k}
\br_k := \frac{\psi(n_{k+1})}{\sqrt{n_{k+1}} }k^2.
\end{equation}
Note that $\br_k$ in \eqref{eq:br_k} is slightly different from \eqref{eq:validity-br}.

For the rest of this section, we check the growth of various quantities along the cycles.
Let $\omega\in \Omega_C$.  For sufficiently large $n$, 
\begin{equation}
\label{eq:xn2bound0}
|x_n|\le c_n \le C \frac{A_n}{\psi(A_n^2)^3}.
\end{equation}
Furthermore $A_n^2 = A_{n-1}^2 + x_n^2$.  This allows us to bound $x_n^2$ and $A_n^2$  in terms of $A^2_{n-1}$.
By squaring  \eqref{eq:xn2bound0} 
we have
\begin{equation} 
\label{eq:xn2bound}
		x_n^2 \le C^2 \frac{A_{n-1}^2}{\psi(A_n^2)^6-C^2}
\end{equation}
and
\begin{equation}
\label{eq:An2bound}
	A_n^2 = A_{n-1}^2 + x_n^2 \le A_{n-1}^2 ( 1+ \frac{C^2}{\psi(A_n^2)^6-C^2}) = 
A_{n-1}^2 \frac{\psi(A_n^2)^6}{\psi(A_n^2)^6-C^2}.
\end{equation}
Since $\psi(A_n^2)^6/(\psi(A_n^2)^6-C^2) \rightarrow 1$ as $n\rightarrow \infty$, we have
%
\begin{equation*}
\lim_{n\rightarrow\infty} \frac{A_n^2}{A_{n-1}^2} = 1.
\end{equation*}
Note that $A_{\tau_k-1}^2 < n_k \le A_{\tau_k}^2$ by the definition of $\tau_k$. Hence for $\omega\in \Omega_C$
we also have
\begin{equation}
\label{eq:stop-error}
\lim_{k\rightarrow\infty}\frac{A_{\tau_k}^2}{n_k} = 1.
\end{equation}
The limits in the following lemma will be useful for our argument.
\begin{lemma}
\label{lem:cycle-growth}
For $\omega\in \Omega_C$
\begin{equation}
	\lim_{k \rightarrow \infty}	\frac{\uc(n_k)}{\psi(n_{k+1})} =1, \quad 
	\lim_{k \rightarrow \infty} \frac{k^5A_{\tau_k}^2}{n_{k+1}} = e^{-5},\quad
	\lim_{k \rightarrow \infty} \br_k A_{\tau_k}\psi(n_{k+1})= 0. \quad
\end{equation}
\end{lemma}
\begin{proof}
	All of $\uc(n_k)$, $\uc(n_{k+1})$, $\lc(n_k)$, $\lc(n_{k+1}), \psi(n_{k+1}), \psi(n_{k+1} /k^4)$ 
	are of the order 
	\begin{equation}
	\label{eq:order-of-nk}
	\sqrt{2 \ln \ln e^{5k\ln k}}(1+o(1))=\sqrt{2\ln k}(1+o(1))
	\end{equation}
	as $k\rightarrow\infty$ and the first equality holds by \eqref{eq:uc0}.
	The second equality holds by \eqref{eq:stop-error} and
	\begin{align*}
    \lim_{k\rightarrow\infty} \frac{k^5 n_k}{n_{k+1}}=
	\lim_{k\rightarrow\infty} \frac{k^{5(k+1)}}{(k+1)^{5(k+1)}} =\lim_{k\rightarrow\infty} \left(1-\frac{1}{k+1}\right)^{5(k+1)}
	= e^{-5}.
	\end{align*}
	Then $A^2_{\tau_k}/n_{k+1}=(1+o(1))n_k/n_{k+1}=O(k^{-5})$ and the third equality holds by
	\begin{align*}
	\br_k A_{\tau_k}\psi(n_{k+1}) \le \psi(n_{k+1})^2k^2 ((1+\delta)n_k/n_{k+1})^{1/2} \rightarrow 0 \qquad(k\rightarrow\infty).
	\end{align*}
\end{proof}

\subsection{Stopping times for aborting and sequential  freezing for each cycle}
In \eqref{eq:bssd} of the next section we will introduce another capital process 
$\bssd_n^{\br_k,k}$, which will be employed in each cycle.
Here we introduce some stopping times for aborting the cycle and for 
sequential freezing of accounts in $\bssd_n^{\br_k,k}$.

We say that we {\em abort} the $k$-th cycle, 
when we freeze all accounts in the $k$-th cycle and wait for the $(k+1)$-st cycle.
There are two cases for aborting the $k$-th cycle. 
The first case is when some $c_n$ is too large for $\omega\in \Omega_{C}$. 
Define
\begin{equation}
\label{eq:sigma-kC}
\sigma_{k,C} := \min \left \{n \ge \tau_k \mid c_n \psi(A^2_{\tau_{k}})^3  > (1+\delta)C A_{n-1}\right \}.
\end{equation}
We will abort the $k$-th cycle if $\sigma_{k,C} < \tau_{k+1}$.
Note that for $\omega \in \Omega_C$, there exists $k_1(\omega)$ such that
\begin{equation}
\label{eq:simga-kC-infty}
 \sigma_{k,C}=\infty,\  \ \text{for}  \ \ k \ge k_1(\omega).
\end{equation}

Another case is when $S_n$ is too large.  Define
\begin{align}
\label{eq:nu_k}
\nu_k := \min\{n \ge \tau_k  \mid A_{n} \psi(A^2_{n}) <  S_{n} \}.
\end{align}
If $\nu_k < \tau_{k+1}$, then Skeptic is happy to abort the $k$-th cycle, because he wants to force
$S_n \ge A_{n} \psi(A^2_{n})\ i.o.$  \ The above two stopping times will be used in the final construction 
of a dynamic strategy in Section \ref{subsec:skepticforcesharpness}.

For each cycle, we define another family of stopping times indexed by $w=1,\ldots, \lceil \ln k\rceil$,
by
\begin{align}
\label{eq:tau-k-w}
\tau_{k,w}:=
\min \left\{n \mid A^2_n \ge e^{2(w+2)}  \frac{n_{k+1}}{k^4}\right\}. 
\end{align}
for sequential freezing of accounts of $\bssd_n^{\br_k,k}$ in \eqref{eq:bssd}.
We have $\tau_k \le \tau_{k,w}$ for $k \ge 1$ and $w\ge 1$, because
\[
\frac{n_{k+1}}{k^4} = \frac{(k+1)^{5(k+1)}}{k^4} > k^{5k}=n_k.
\]

\begin{lemma}
	\label{lemm:all-accounts-stop}
    Let $\omega \in \Omega_C$. 
$	\tau_{k,\lceil \ln k \rceil} \le \tau_{k+1}$ for sufficiently large $k$.
\end{lemma}

\begin{proof}
	By
$	A^2_{\tau_{k,w}-1} \le e^{2(w+2)}n_{k+1} / k^4$
    and by \eqref{eq:xn2bound}, for sufficiently large $k$ we have
	\begin{align*}
	x_{\tau_{k,w}}^2 \le (1+\delta)C^2 \frac{A^2_{\tau_{k,w}-1} }{\psi(A^2_{\tau_{k}})^6} \le \frac{(1+\delta) C^2 }{\psi(A^2_{\tau_{k}})^6} \times \frac{e^{2(w+2)} n_{k+1}}{k^4}
	\end{align*}
	and
	\begin{align}
	\label{eq:A_n-stop-above2}
	A_{\tau_{k,w}}^2 \le A^2_{\tau_{k,w}-1} + x^2_{\tau_{k,w}}
	&\le  (1+\delta) e^{2(w+2)}  \frac{n_{k+1}}{k^4}.
	\end{align}
Then
	\begin{align*}
		A^2_{\tau_{k, \lceil \ln k \rceil } } \le (1+\delta) \left(e^{2(\ln k+2)} \frac{n_{k+1}}{k^4}\right) =
		 (1+\delta) e^4 \frac{n_{k+1}}{k^2} \le n_{k+1} \le A^2_{\tau_{k+1}}.
	\end{align*}
\end{proof}

We also compare $\tau_{k,w}$ to $\sigma_{\br_k e^{-w+1}}$ defined in \eqref{eq:stop}.
This is needed for applying the bounds derived in previous sections to 
$\bssd_n^{\br_k,k}$ in the next section. 
\begin{lemma}
\label{lemm:sigma_br-sigma_c}	
Let $\omega\in \Omega_C$. 
$\tau_{k,w}\le \sigma_{\br_k e^{-w+1}}$ 
for sufficiently large $k$. 
\end{lemma}

\begin{proof}
By \eqref{eq:A_n-stop-above2} and by Lemma
\ref{lem:bar_c_n_bound}, 
for sufficiently large $k$
\begin{align*}
\br_{k}e^{-w+1} \bar c_{\tau_{k,w}}
 \le \frac{\psi(n_{k+1})}{\sqrt{n_{k+1}}}k^2 e^{-w+1}\times
(1+\delta)^2 C  \frac{e^{w+2}   \sqrt{n_{k+1}} }{k^2 \psi(A^2_{\tau_k})^3} 
\le 
(1+\delta)^2 C e^3 \frac{\psi(n_{k+1})}{\psi(A^2_{\tau_k})^3} \le \delta,
\end{align*}
because 
$\psi(n_{k+1})/\psi(A^2_{\tau_k})^3 \rightarrow 0$ as $k\rightarrow\infty$
by \eqref{eq:order-of-nk}.
\end{proof}

\subsection{Further discrete mixture of processes for each cycle with sequential freezing}
\label{subsec:Futher discrete mixture of processes}
We introduce another discrete mixture of capital process for the $k$-th cycle. Define
\begin{align}
	\label{eq:bssd}
	\bssd_n^{\br_k,k}:= \frac{1}{\lceil \ln k \rceil}\sum_{w=1}^{\lceil \ln k \rceil} \bss_{\min(n,\tau_{k,w})}^{\br_k e^{-w}}  =  \frac{1}{\lceil \ln k \rceil}\sum_{w=1}^{\lceil \ln k \rceil} (2\mcps_{\min(n,\tau_{k,w})}^{\br_k e^{-w}}
	- \cps_{\min(n,\tau_{k,w})}^{\br_k e^{-w+1}}).
\end{align}
Note that  the $w$-th account in the sum of $\bssd_n^{\br_k,k}$is frozen at the stopping time $\tau_{k,w}$.
This is needed since the bound for $c_n$ is growing even during the $k$-th cycle. 

In order to bound $\bssd_n^{\br_k,k}$, we first 
bound $C_1$ in \eqref{eq:case-i-negative-condition-1} for each $w$ in the sum of \eqref{eq:bssd}
by a constant independent of $n$. Note that we only need to consider $n\le \tau_{k,w}$ for the $w$-th account.
\begin{lemma}
	\label{lem:C_1-const}
	Let $\omega \in \Omega_C$. 
        $(\br_k e^{-w})^3 A_n^2 \bar c_n$  and hence $C_1$ 
        are bounded from above by
	\begin{align}
		\label{eq:remainder-const}
                (\br_k e^{-w})^3 A_n^2 \bar c_n &\le (1+\delta)^5 C e^6,\\
		\label{eq:C_1-const}
		C_1 &\le 2 e^{(1+\delta)^5Ce^6} \exp\left(\frac{(2e-1)((1+\delta)^5Ce^6(1+e^3)+ \ln 2)}{(e-1)^2}\right)=:\bar C_1, 
	\end{align}
for sufficiently large $k$.
\end{lemma}

\begin{proof}
%
By \eqref{eq:order-of-nk}, for  sufficiently large $k$
\begin{align}
\label{eq:psi-ratio}
 \frac{\psi(n_{k+1})}{\psi(A^2_{\tau_{k,w}})} \le 
\frac{\psi(n_{k+1})}{\psi(n_k)}  \le 1+\delta.
\end{align}
Thus
\begin{align*}
\br_k^3 e^{-3w}  A_{\min(n,\tau_{k,w})}^2 \bar c_{\min(n,\tau_{k,w})}
&\le 
\br_k^3 e^{-3w}  \times A_{\tau_{k,w}}^2 \times \bar c_{\min(n,\tau_{k,w})} \\
&\le
\frac{\psi(n_{k+1})^3}{n_{k+1}^{3/2}}k^6 e^{-3w}  \times A^2_{\tau_{k,w}} \times (1+\delta)C \frac{A_{\tau_{k,w}}}{\psi(A^2_{\tau_k})^3} \\
&\le
(1+\delta)C \frac{\psi(n_{k+1})^3}{\psi(A^2_{\tau_k})^3} k^6e^{-3w} \frac{A^3_{\tau_{k,w}}}{n^{3/2}_{k+1}} 
\le 
(1+\delta)^5Ce^{6}.
\end{align*}
\end{proof}

\begin{lemma}
	\label{lem:sn-bssd-bound}
	Let $\omega \in \Omega_C$. For sufficiently large $k$,
	\begin{align}
		\label{eq:sn-bssd-bound}
		\bssd_n^{\br_k,k} \le \bar C_1 +  \frac{2}{\lceil \ln k \rceil} e^{(1+\delta)^5 C e^6}
		\max_{\br\in [\br_k/k,\br_k]} \left(\min \{ e^{S_n^2/(2n)} \frac{\sqrt{2\pi}}{\br A_n}, 
e^{\br S_n} \}\right), \quad n \in [\tau_k,\tau_{k+1}],
	\end{align}
	 where $\bar C_1$ is given by the right-hand side of \eqref{eq:C_1-const}.
\end{lemma}

\begin{proof}
	We have $|\br_k e^{-w} \bar c_{\min(n,\tau_{k,w})}| \le |\br_k e^{-w+1} \bar c_{\min(n,\tau_{k,w})}| \le \delta  $ 
	by Lemma \ref{lemm:sigma_br-sigma_c}.
	Then we can complete the proof of \eqref{eq:sn-bssd-bound} by Lemma \ref{lem:bss-bound} and Lemma \ref{lemm:sigma_br-sigma_c}	
        because the
	length of the interval
	\begin{align*}
	\left \{ w \mid \frac{S_n}{ne} < \br e^{-w} < \frac{S_n e}{n}\right \}
	\end{align*}
	is equal to $2$.
\end{proof}

As in Chapter 5 of Shafer and Vovk \cite{ShaferVovk2001Probability}, we use  $\bssd_n^{\br_k,k}$ in the following form.
\begin{align}
	\label{eq:dynp}
	\dynp_{n}^{\br_k,D} :=\alpha + \frac{1}{D}  \lceil \ln k \rceil \psi(n_{k+1}) e^{-\psi(n_{k+1})^2/2} (\alpha-\bssd_{n-\tau_k}^{\br_k,k}),\quad \alpha = 1-\frac{2}{e}, \,\, D=\frac{24\sqrt{2\pi}e^{(1+\delta)^5e^6C}+4\bar C_1}{\alpha}.
\end{align}
Here we give a specific value of $D$ for definiteness, but from the proof below it will be clear that any sufficiently large $D$ can be used. 
Since the strategy for $\bssd_{n-\tau_k}^{\br_k,k}$ is applied only to $x_n$'s in the cycle, $\alpha=\dynp_{\tau_k}^{\br_k,D}=\bssd_{0}^{\br_k}$.
Concerning $\dynp_{n}^{\br_k,D}$ we prove the following two propositions.


\begin{proposition}
	\label{prop:dynp}
	Let $\omega \in \Omega_C$. Suppose that
	\begin{align}
		\label{eq:within-range}
		-A_n \uc(A_n^2) \le S_n \le A_n \psi(A_n^2), \qquad  \forall n\in [\tau_k,\tau_{k+1}].
	\end{align}
	and $\tau_{k+1} < \sigma_{k,C}$.
	Then for sufficiently large $k$
	\begin{align}
		\label{eq:nkn-lower-bound}
		\dynp_n^{\br_k,D} \ge \frac{\alpha}{2}, \qquad \forall  n\in [\tau_k,\tau_{k+1}],
	\end{align}
	and
	\begin{align}
		\label{eq:nk+1} 
	    \dynp_{\tau_{k+1}}^{\br_k,D} \ge \alpha \left( 1 + \frac{1-\delta}{D} \lceil \ln k \rceil  \psi(n_{k+1}) e^{-\psi(n_{k+1})^2/2} \right).
	\end{align}
\end{proposition}

\begin{proof}
	In our proof we denote $t=n-\tau_k$, $S_t = S_n - S_{\tau_k}$  and 
$A_t^2 = A_n^2-A^2_{\tau_k}$ for $n>\tau_k$.
	For proving \eqref{eq:nkn-lower-bound}, we use \eqref{eq:sn-bssd-bound} for $S_t$.
	We bound $\bssd_{t}^{\br_k,k}$ from above.
	By the term $\dfrac{2}{\lceil \ln k \rceil}$ on the right-hand side of \eqref{eq:sn-bssd-bound}, it suffices to show
	\begin{align*}
		&S_t \le A_{\tau_k} \uc(A^2_{\tau_k}) + \sqrt{A^2_{\tau_k} + A_t^2} \psi(A^2_{\tau_k}+ A_t^2)  \\
		& \quad \Rightarrow \  
		\psi(n_{k+1}) e^{-\psi(n_{k+1})^2/2} 
		2e^{(1+\delta)^5e^6C}  \min \{ e^{S_t^2/(2A_t^2)} \frac{\sqrt{2\pi}}{\br A_t}, e^{\br S_t} \}\le  \frac{D\alpha}{4}, 
		\,\forall \br \in [\br_k/k, \br_k],\,\forall t\in [0,\tau_{k+1}-\tau_k]
	\end{align*}
	for sufficient large $k$.
%
	Let 
	\begin{align}
		\label{eq:c1-const}
		c_1 = \frac{9}{(1+2\delta)^2} \quad \mathrm{s.t.} \quad \frac{1}{2} - \frac{1}{\sqrt{c_1}} - \delta > 0.
	\end{align}
	We distinguish two cases:
	\[
	\text{(a)} \ A_t^2\le \frac{\psi(n_{k+1})^2}{c_1 \br^2}, \quad \text{(b)}\  \frac{\psi(n_{k+1})^2}{c_1 \br^2} < A_t^2 \le A^2_{\tau_{k+1}}-A^2_{\tau_k}.
	\] 
	
	For case (a), $A_{\tau_k} \psi^{U}(A_{\tau_k}^2) \le (1+\delta) A_{\tau_k}\psi(n_{k+1})$ by the first equality in Lemma \ref{lem:cycle-growth} for sufficiently large $k$. 
	Also $\psi(A^2_{\tau_k}+A_t^2)\le \psi(n_{k+1})$. Hence in this case
	\begin{align*}
		\br S_t \le \left((1+\delta) \br A_{\tau_k} + \sqrt{\br^2 A^2_{\tau_k} +\psi(n_{k+1})^2/c_1 } \right)\psi(n_{k+1}).
	\end{align*}
	Then for $\br \le \br_k$ by the third equality in Lemma \ref{lem:cycle-growth}
	\begin{align}
		\label{eq:another-large-k}
		\br S_t \le  \left((1+\delta)\br_k A_{\tau_k}+ \sqrt{\br_k^2 A^2_{\tau_k} + \psi(n_{k+1})^2/c_1}\right)\psi(n_{k+1})
		= \psi(n_{k+1})^2 \left(\frac{1}{\sqrt{c_1}}+\delta\right)
	\end{align}	
	for sufficiently large $k$. Since
	\begin{align*}
		\psi(n_{k+1}) e^{-\psi(n_{k+1})^2/2} 2 e^{(1+\delta)^5e^6C} e^{\br S_t} 
		\le \psi(n_{k+1}) \exp\left(-\psi(n_{k+1})^2\big(\frac{1}{2} - \frac{1}{\sqrt{c_1}}-\delta\big)\right)2  e^{(1+\delta)^5e^6C}
		\rightarrow 0\quad (k \rightarrow \infty),
	\end{align*}
	we have $\dynp_n^{\br_k,D}\ge \alpha / 2$ uniformly in $\br\in [\br_k/k,\br_k]$.

	For case (b), $\psi(n_{k+1})/\sqrt{c_1} < \br A_t$ and $S_t \le \left( (1+\delta)A_{\tau_k} + \sqrt{A^2_{\tau_k} + A_t^2} \right) \psi(n_{k+1})$. Hence
	\begin{align}
		&\psi(n_{k+1}) 
		e^{-\psi(n_{k+1})^2/2} \times 
		2 e^{(1+\delta)^5e^6C}e^{S_t^2/(2A_t^2)} \frac{\sqrt{2\pi}}{\br A_t} \nonumber \\
		&\le
		\psi(n_{k+1}) e^{-\psi(n_{k+1})^2/2} \times \frac{2e^{(1+\delta)^5e^6C}\sqrt{2\pi}\sqrt{c_1}}{\psi(n_{k+1})}
		\exp\left(\frac{\left((1+\delta)A_{\tau_k}+ \sqrt{A^2_{\tau_k} + A_t^2}\right)^2}{2A_t^2}\psi(n_{k+1})^2\right) \nonumber \\
		&= 
		2e^{(1+\delta)^5e^6C}\sqrt{2\pi}\sqrt{c_1} \exp\left( \frac{(1+(1+\delta)^2)A_{\tau_k}^2 + 2(1+\delta) A_{\tau_k}\sqrt{A^2_{\tau_k}+A_t^2}}{2A_t^2} \psi(n_{k+1})^2\right).
\label{eq:case-b-constant-bound}
	\end{align}
	For $\br\le \br_k$,
	\[
	\frac{\psi(n_{k+1})^2}{c_1 \br^2} <  A_t^2 \ \Rightarrow \ \frac{A^2_{\tau_k}}{A_t^2} \psi(n_{k+1})^2 < c_1 \br^2 A^2_{\tau_k}
	\le c_1 \br_k^2 A^2_{\tau_k} = c_1 \frac{A^2_{\tau_k}}{n_{k+1}}k^4\psi(n_{k+1})^2
=O(k^{-1}\ln k).
	\]
	Hence $\psi(n_{k+1})^2 A^2_{\tau_k}/A_t^2 \rightarrow 0$ as $k\rightarrow \infty$.  Similarly $\psi(n_{k+1})^2 A_{\tau_k} / A_t \rightarrow 0$ as $k\rightarrow\infty$, because
$\psi(n_{k+1})^2 A_{\tau_k} / A_t=O(k^{-1/2}(\ln k)^{3/2})$.
Therefore the right-hand side of \eqref{eq:case-b-constant-bound} is bounded from above by
$2e^{(1+\delta)^5e^6C}\sqrt{2\pi}\sqrt{c_1}(1+\delta)$ for sufficiently large $k$ and
	\begin{align*}
		\psi(n_{k+1}) e^{-\psi(n_{k+1})^2/2} \times 
		2 e^{(1+\delta)^5e^6C}e^{S_t^2/(2A_t^2)} \frac{\sqrt{2\pi}}{\br A_t}  \le 
		\frac{D\alpha}{4},
	\end{align*}
	with the choice of $D$ in \eqref{eq:dynp} and $c_1$ in \eqref{eq:c1-const}. This proves \eqref{eq:nkn-lower-bound}.

	Now we prove \eqref{eq:nk+1}. We focus on the $w$-th account when $n \ge \tau_{k,w}$.
        Recall that in this proof we have been denoting $A_t^2 = A_n^2 - A_{\tau_k}^2$. Similarly we denote $A^2_{\tau_{k,w}}$ instead of $A^2_{\tau_{k,w}}-A^2_{\tau_k}$. Thus 
	\begin{align}
		\label{eq:An-dy-above}
		e^{2(w+2)} \frac{n_{k+1}}{k^4} -A^2_{\tau_k}\le A^2_{\tau_{k,w}}.
	\end{align}
	We will show that $\limsup_{k\rightarrow\infty} \bssd_{\tau_{k+1}-\tau_k}^{\br_k,k} \le 0$, if
	\begin{align}
		\label{eq:sn-dy-above}
		S_{\tau_{k,w}} \le A_{\tau_k} \psi(A^2_{\tau_k}) + A_{\tau_{k,w}} \psi(A^2_{\tau_{k,w}}) \le \psi(n_{k+1}) \left\{A_{\tau_k} + A_{\tau_{k,w}}\right\} \le 2 \psi(n_{k+1}) A_{\tau_{k,w}}.
	\end{align}

	We evaluate 
\[
\mcps_{\tau_{k,w}}^{\br_ke^{-w},k}:=\int_{2/e}^1 \exp \left(u\br_k e^{-w} S_{\tau_{k,w}}-u^2 \br_k^2 e^{-2w}A^2_{\tau_{k,w}}/2\right)du
\]
from above. Because $u\br_k e^{-w} S_{\tau_{k,w}}-u^2 \br_k^2 e^{-2w} A^2_{\tau_{k,w}}/2$ is maximized at $u= S_{\tau_{k,w}}/(\br_k e^{-w} A^2_{\tau_{k,w}})$ and 
	\begin{align*}
		\frac{S_{\tau_{k,w}}}{\br_k e^{-w} A^2_{\tau_{k,w}}}
\le \frac{2 \psi(n_{k+1})A_{\tau_{k,w}}} { (\psi(n_{k+1}) k^2/\sqrt{n_{k+1}}) e^{-w} A^2_{\tau_{k,w}}}
\le \frac{2\sqrt{n_{k+1}}}{k^2 e^{-w} A_{\tau_{k,w}}} \le \frac{2}{e^2}\le \frac{2}{e},
	\end{align*} 
	the integrand in $\mcps_{\tau_{k,w}}^{\br_ke^{-w},k}$ is maximized at $2/e$ and we have
	\begin{align*}
		\mcps_{\tau_{k,w}}^{\br_ke^{-w},k}
		&\le
                \exp \left(\frac{2}{e}\br_k e^{-w} S_{\tau_{k,w}}-\frac{2 \br_k^2 e^{-2w}A^2_{\tau_{k,w}}}{e^2}\right). 
	\end{align*}

	By \eqref{eq:An-dy-above} and \eqref{eq:sn-dy-above}, for sufficiently large $k$,
	\begin{align*}
		\frac{2}{e}\br_k e^{-w} S_{\tau_{k,w}}-\frac{2 \br_k^2 e^{-2w}A^2_{\tau_{k,w}}}{e^2} 
		&\le \frac{4 \br_k \psi(n_{k+1})A_{\tau_{k,w}}}{ e^{w+1}} - \frac{2\br_k^2 A^2_{\tau_{k,w}}}{e^{2(w+1)}}\nonumber  \\
		&= \frac{\psi(n_{k+1})^2k^2A_{\tau_{k,w}}}{\sqrt{n_{k+1}}e^w} \left( \frac{4}{e} -  \frac{2k^2A_{\tau_{k,w}}}{e^2\sqrt{n_{k+1}}e^w}\right)\nonumber \\
		&\le \frac{\psi(n_{k+1})^2k^2A_{\tau_{k,w}}}{\sqrt{n_{k+1}}e^w} 
		\left( \frac{4}{e} - \frac{2}{e^2}\sqrt{e^4-\frac{(1+\delta)k^4n_k}{n_{k+1}e^{2w}}}\right) \\
		&\le -\psi(n_{k+1})^2 \frac{k^2}{\sqrt{n_{k+1}}e^w} \times \frac{\sqrt{n_{k+1}}e^{w+2}}{k^2} \times \frac{1}{2}
\qquad 
\nonumber  \\
		&= - \frac{e^2\psi(n_{k+1})^2}{2}.
	\end{align*}
The last inequality holds because $\lim_{k\rightarrow\infty}k^4 n_k/n_{k+1} =0$
and $4/e - 2 < -1/2$. Hence 
$\mcps_{\tau_{k,w}}^{\br_k e^{-w},k} \rightarrow 0$ uniformly in $1\le w \le \lceil \ln k \rceil$.
This implies $\limsup_{k\rightarrow\infty} \bssd_{\tau_{k+1}-\tau_k}^{\br_k,k} \le 0$.
\end{proof}

\begin{proposition}
	\label{prop:dynp2}
	Let $\omega \in \Omega_C$.
	Suppose that $\nu_k \le \min(\tau_{k+1}, \sigma_{k,C})$ and
\[
	-A_{n} \uc(A_{n}^2) \le S_{n},  
        \,\, \forall n\in [\tau_k,\nu_k].
\]
        Then for sufficiently large $k$ 
	\begin{align*}
	\dynp_{\nu_k}^{\br_k,D} \ge \frac{\alpha}{2} .
	\end{align*}
\end{proposition}

\begin{proof}
  As in the proof of the previous lemma, we denote  $t= n-\tau_k$, $S_t = S_n - S_{\tau_k}$  and    $A_t^2 = A_n^2-A^2_{\tau_k}$.
		We  distinguish two cases:
		\[
		\text{(a)} \ A_{\nu_k}^2\le \frac{\psi(n_{k+1})^2}{c_1 \br^2}, \quad \text{(b)}\  \frac{\psi(n_{k+1})^2}{c_1 \br^2} < A_{\nu_k}^2 \le A^2_{\tau_{k+1}}-A^2_{\tau_k}.
		\] 
		For case (a), 
		for sufficiently large $k$ and for any $\br \le \br_k$, as in \eqref{eq:another-large-k},
		\begin{align*}
		\br S_{\nu_k}
		&\le
	    \br \left(S_{\nu_k-1}+c_{\nu_k} \right)
	    \le \br \left(\left((1+\delta) A_{\tau_k} + \sqrt{A^2_{\tau_k} + A^2_{\nu_k -1 }} \right)\psi(n_{k+1})+(1+\delta)C\frac{\sqrt{A^2_{\tau_k}+A^2_{\nu_k-1}}}{\psi(A^2_{\tau_k})^3} \right)\\
&\le \psi(n_{k+1})^2\left(\frac{1}{\sqrt{c_1}}+\delta\right) 
		\end{align*}
and 
		\begin{align*}
		\psi(n_{k+1}) e^{-\psi(n_{k+1})^2/2} 2 e^{(1+\delta)^5 e^6C} e^{\br S_{\nu_k}} 
	    \rightarrow 0\quad (k \rightarrow \infty).
		\end{align*}
		Hence  $\dynp_{\nu_k}^{\br_k,D} \ge \alpha/2$ uniformly in $\br\in [\br_k/k,\br_k]$.
	
		For case (b), $S_{\nu_k}$ can be evaluated as
		\begin{align*}
		S_{\nu_k} 
		&\le 
		S_{\nu_k-1}+c_{\nu_k} \le S_{\nu_k-1}+(1+\delta)C\frac{\sqrt{A_{\tau_k}^2 + A_{\nu_k-1}^2}}{\psi(A^2_{\tau_k})^3}\\ \nonumber
		&\le
		\left((1+\delta)A_{\tau_k} + \sqrt{A^2_{\tau_k}  + A_{\nu_k}^2}\right) \psi(n_{k+1}) +(1+\delta)C\frac{\sqrt{A_{\tau_k}^2 + A_{\nu_k}^2}}{\psi(A^2_{\tau_k})^3}\\ \nonumber
		 & \le
		 \left((1+\delta)A_{\tau_k}  + \sqrt{A^2_{\tau_k}  + A_{\nu_k}^2}\left(1+\frac{(1+\delta)C}{\psi(A^2_{\tau_k})^3\psi(n_{k+1})} \right) \right) \psi(n_{k+1})
		\end{align*}
by \eqref{eq:psi-ratio}.
Put 
\[
q_k^2 := \frac{A^2_{\tau_k}}{ A_{\nu_k}^2} \le \frac{c_1\br_k^2}{\psi(n_{k+1})^2}, \qquad s_k: = \frac{(1+\delta)C}{\psi(A^2_{\tau_k})^3\psi(n_{k+1})},
\]
so that 
$\lim_k q_k \psi(n_{k+1})^2=0$ and $\lim_k s_k\psi(n_{k+1})^2=0$. Then for sufficiently large $k$
		\begin{align*}
		\frac{S_{\nu_k}^2}{2A_{\nu_k}^2}
		 &\le \left(( 1+\delta)^2 \frac{q_k^2}{2} + (1+\delta)(1+s_k) q_{k} \sqrt{1+q^2_{k} }
		  + (1+ s_k)^2\left(\frac{1}{2} + \frac{q^2_{k} }{2}\right) \right) \psi(n_{k+1})^2\nonumber \\
		 &\le \frac{\psi(n_{k+1})^2}{2} + \delta .
		\end{align*}
		Then
		\begin{align*}
		\psi(n_{k+1}) e^{-\psi(n_{k+1})^2/2} \times 
		2 e^{(1+\delta)^5 e^6C}e^{S_{\nu_k}^2/(2A_{\nu_k}^2)} \frac{\sqrt{2\pi}}{\br A_{\nu_k}} 
		 \le 
		2 e^{(1+\delta)^5 e^6C+ \delta} \sqrt{2\pi c_1} e^\delta
		\le  \frac{D\alpha }{4}.
		\end{align*}
\end{proof}

\subsection{Dynamic strategy forcing the sharpness}
\label{subsec:skepticforcesharpness}

Finally, we prove Proposition \ref{th:self-normalized-efkp-lil-dash}.
We assume that by the validity result, Skeptic already employs a strategy forcing
$S_n \ge -A_n\uc(A_n^2)\ a.a.$  for $\omega\in\Omega_{C}$.
In addition to this strategy, based on Proposition \ref{prop:dynp}, 
consider the following strategy.
\begin{quote}
	Start with initial capital $\cps_0=\alpha$.\\
	Set $k=1$.\\
	Do the followings repeatedly:\\
	\indent 1.  Apply the strategy in Proposition \ref{prop:dynp} for $n\in [\tau_k, \tau_{k+1}]$. \\ 
	\indent \quad If $\tau_{k+1} < \min(\sigma_{k,C} , \nu_k)$, then go to 2. Otherwise go to 3.\\
	\indent 2. Let $k=k+1$. Go to 1.\\
	\indent 3. 
Wait until $\exists k'$ such that $-\sqrt{\tau_{k'}} \uc(\tau_{k'}) \le S_{\tau_{k'}} \le
\sqrt{\tau_{k'}}\psi(\tau_{k'})$.
Set $k=k'$ and go to 1.
\end{quote}

By this strategy  Skeptic keeps his capital non-negative for every path $\omega$.
For $\omega\in \Omega_0$, $\tau_k=\infty$ for some $k$ and Skeptic stays in Step 1 forever.
For $\omega\in \Omega_{=\infty}$, Step 3 is performed infinite number of times, but 
the overshoot of $|x_n|$ in Step 3 does not make Skeptic bankrupt by Proposition \ref{prop:dynp2}. 
Now consider $\omega\in \Omega_{C}$.
Since Skeptic already employs a strategy forcing 
$S_n \ge -A_n \uc(A^2_n)\ a.a.$, the lower bound in \eqref{eq:within-range} violated only
finite number of times. 
By $\omega \in \Omega_{C}$, $n \ge \sigma_{k,C}$ is happens  only finite number of times. 
Hence if $S_n \le A_n\psi(A^2_n) \ a.a.$, then Step 3 is performed only finite number of times
and there exists $k_0$ such that only Step 2 is repeated for all $k\ge k_0$.
Now for each iteration of Step 2, Skeptic
multiplies his capital at least by 
\begin{align*}
1 + \frac{1-\delta}{D} \lceil \ln k \rceil \psi(n_{k+1}) e^{-\psi(n_{k+1})^2/2}.
\end{align*}
Then
\begin{align*}
\frac{1-\delta}{D} \sum_{k=k_0}^\infty \lceil \ln k \rceil \psi(n_{k+1}) e^{-\psi(n_{k+1})^2/2}
\le \prod_{k=k_0}^\infty  \left( 1 + \frac{1-\delta}{D} \lceil \ln k \rceil\psi(n_{k+1}) e^{-\psi(n_{k+1})^2/2}\right).
\end{align*}
Since the left-hand side diverges to infinity, the above strategy forces the sharpness.

\bibliographystyle{abbrv}
\bibliography{sn-efkp.bib}

\end{document}